\documentclass[12pt]{article}
\usepackage[T1]{fontenc}
\usepackage{amsmath,amssymb,amsthm,amscd,dsfont}
\usepackage{amsfonts,latexsym,rawfonts,amsmath,amssymb,amsthm}
\usepackage{lscape}
\usepackage{amscd, float,times,rotating}
\usepackage{pb-diagram}
\usepackage{hyperref}
\numberwithin{equation}{section}

\newtheorem{prop}{Proposition}[section]
\newtheorem{theorem}[prop]{Theorem}

\newtheorem{remark}[prop]{Remark}
\newtheorem{example}[prop]{Example}
\newtheorem{definition}[prop]{Definition}

\newcommand{\R}{\mathbb{R}}

\begin{document}

\title{Estimates of $p$-capacity for manifolds with Ricci curvature bounded from below}
\author{Xiaoshang Jin\thanks{X. J. is supported by ``the Fundamental Research Funds for the Central Universities'', HUST: \# 2025BRSXB002 and NSFC (Grant No. 12471054)} \and Zhiwei L\"u \and Jiabin Yin\thanks{J. Y. is supported by the NSFC (Grant No. 12201138), Mathematics Tianyuan fund project (Grant No. 12226350) and NSF of Henan Province (Grant No. 262300421869)}}
\date{}
\maketitle

\begin{abstract}
We study sharp estimates for the \(p\)-capacity on complete non-compact Riemannian manifolds under lower Ricci curvature bounds.

First, we establish sharp comparison inequalities for the \(p\)-capacity of bounded smooth domains in manifolds satisfying \(\operatorname{Ric}\ge -ng\). The estimates are expressed in terms of the boundary mean curvature and correspond to natural warped-product model ends. We characterize all equality cases and show that equality forces the exterior region to be isometric to the corresponding warped product. We also obtain an analogous sharp estimate under nonnegative Ricci curvature, whose equality case is described by an asymptotically flat model end.

Second, we investigate normalized lower bounds for the relative \(p\)-capacity of condensers. We introduce scale-invariant quantities involving the volume of the inner set and the diameter of the ambient domain, establish uniform positive lower bounds, and determine the optimal ranges of the normalization parameters.
\end{abstract}

\section{Introduction}

\begin{definition}
  Assume that $(M,g)$ is a complete Riemannian manifold and $\Omega\subseteq M$ is a compact subset. For any $p>1,$
the $p$-capacity of $\Omega$ in $M$ is defined as the infimum of the $p$-Dirichlet energy:
$$\mathrm{Cap}_p(\Omega)=\mathrm{Cap}_p(\Omega,M)=\inf\left\{\int_{M}|\nabla u|^p \,dv_g:\ u|_{\Omega}=1, \ u\in C^\infty_0(M)\right\}$$
\end{definition}
In general, we require that $M$ is non-compact and $\Omega$ is a closed domain with compact smooth boundary $\partial\Omega.$ Then
$$
\mathrm{Cap}_p(\Omega)=\inf\left\{\int_{M\setminus\Omega}|\nabla u|^p\, dv_g:\ u \ \text{is locally Lipschitz in}\  M\setminus\Omega,\ u|_{\partial\Omega}=1, \ \lim\limits_{x\rightarrow\infty}u(x)=0\right\}
$$
When the absolute $p$-capacitary potential exists, it is the unique solution $u_p$ of the following exterior Dirichlet problem:
$$\begin{cases}
  \Delta_p u={\rm div}(|\nabla u|^{p-2}\nabla u)=0\ & {\rm in}\ M\setminus\Omega, \\
  u=1\ &{\rm on}\ \ \partial\Omega,\\
  u(x)\rightarrow0  &{\rm as} \ \ x\rightarrow\infty.
\end{cases}$$
Such a function $u_p$ is called the $p$-capacitary potential of $\Omega$ in $M$. Its existence requires an appropriate $p$-hyperbolicity assumption; see, for instance, Theorem 4.1 in \cite{fogagnolo2022minimising}. When it exists, the $p$-capacity is given by
$$
\mathrm{Cap}_p(\Omega)=\int_{M\setminus\Omega}|\nabla u_p|^p\, dv_g=\int_{\{u=t\}}|\nabla u_p|^{p-1}\, d\sigma_g
$$
for any regular value $t\in(0,1].$ Here $d\sigma_g$  denotes the ($n-$dimensional) Riemannian surface element induced on the level set.
\par The concept of capacity originates from electrostatics and $\mathrm{Cap}_2(\Omega,\mathbb{R}^3)$ represents the total electrical charge the conductor $\Omega$ can hold. For a general $p>1$, the $p$-capacity is a fundamental notion in nonlinear potential theory and geometric analysis.
In Euclidean space $\mathbb{R}^n$, sharp relationships among capacity, volume, surface area, and mean curvature have been extensively studied; see, for instance, \cite{adams2025full,kruglikov1987capacity,maz2013sobolev,polya1951isoperimetric,xiao2017p}. More recently, sharp upper and lower bounds for the $p$-capacity in hyperbolic space have been obtained by inverse mean curvature flow and conformal methods; see \cite{jin2026sharp,jin2025sharply,li2025sharp} for further details.
\par In general Riemannian manifolds, \cite{grigor1999isoperimetric} provided some basic conceptions and tools of capacity. In 2008, Bray and Miao \cite{bray2008capacity}  established a sharp upper bound for $\mathrm{Cap}_2(\Omega)$
 in an asymptotically flat 3-manifold with nonnegative scalar curvature (other related mass-capacity inequality see \cite{miao2024implications,miao2025mass,oronzio2025area,oronzio2025adm}).  The result links the capacity to the Hawking mass and the total ADM mass and was later extended to general $p\in(1,3)$ case by Xiao in \cite{xiao2016p}. The theory was further developed to mass-$p$-capacity in \cite{mazurowski2023monotone,xia2024new} (other relevant literature
see \cite{yin2026sharp}). For asymptotically hyperbolic Einstein manifolds, the first author studied the capacity of balls in \cite{jin2025relative}.
\par In this paper, we first establish sharp comparison inequalities for the $p$-capacity in terms of the boundary mean curvature on manifolds with Ricci curvature bounded from below. In \cite{jin2024willmore}, the authors recently proved a Willmore-type inequality on complete $(n+1)$-dimensional manifolds satisfying $\mathrm{Ric}\geq-ng$, together with a rigidity statement leading to an asymptotically hyperbolic end when the mean curvature is constant. This geometric picture provides a natural framework for the capacity estimates developed here.

\par To better understand our new results, let us start with the three standard models of the warped products:
\begin{example}
$$\begin{cases}
   (M_1,g_1)=([0,\infty)\times \mathbb{S}^n,\ dr^2+e^{2r} g_{\mathbb{S}^n}) \\
   (M_2,g_2)= ([0,\infty)\times \mathbb{S}^n,\ dr^2+\sinh^2r g_{\mathbb{S}^n})= {\rm hyperbolic\ space}\\
   (M_3,g_3)=([0,\infty)\times \mathbb{S}^n,\ dr^2+\cosh^2r g_{\mathbb{S}^n}).
\end{cases}
$$
A direct calculation indicates that ${\rm Ric}[g_i]\geq -ng_i,\ i=1,2,3.$ For any $p>1,$ consider the radial $p$-harmonic functions respectively:
$$\begin{cases}
  u_p(r)=\int_r^\infty e^{\frac {n t}{1-p}}\,dt=\frac{p-1}{n}e^{\frac {n r}{1-p}} \\
   v_p(r)=\int_r^\infty(\sinh t)^{\frac n{1-p}}\,dt\\
  w_p(r)=\int_r^\infty(\cosh t)^{\frac n{1-p}}\,dt.
\end{cases}
$$ Let $B_i(R)=\{|r|\leq R\}\subseteq M_i \ (i=1,2,3)$ in the three models above for some $R>0,$ then
$$
\begin{cases}
 \mathrm{Cap}_p(B_1(R),M_1)=\omega_n u_p(R)^{1-p}\\
 \mathrm{Cap}_p(B_2(R),M_2)=\omega_n v_p(R)^{1-p}\\
 \mathrm{Cap}_p(B_3(R),M_3)=\omega_n w_p(R)^{1-p}
\end{cases}
$$where $\omega_n={\rm Vol}(\mathbb{S}^n).$
\end{example}
Throughout the paper, ${\rm arc}\,f$ denotes the inverse of a trigonometric or hyperbolic function $f$. The symbols $|\Omega|$ and $|\partial\Omega|$ denote the volume of $\Omega$ and the area of $\partial\Omega$, respectively, while $H$ (or $H(x)$ and $H(x,r)$) denotes the mean curvature with respect to the outward unit normal.
\begin{theorem}\label{mainthm}
 Assume that $p>1$ and $(M,g)$ is a complete non-compact Riemannian manifold of dimension $n+1$ with $\mathrm{Ric}\ge -ng.$ Let $\Omega\subset M$ be a compact smooth domain and $H$ be the mean curvature of $\partial\Omega.$ Then we have the following three sharp estimates:
 \begin{itemize}
   \item [(1)]
   \begin{equation}\label{eq1.1}
          \mathrm{Cap}_p(\Omega)\le \Big(\frac{n}{p-1}\Big)^{p-1}  \int_{\partial \Omega} \Bigl( \max\big\{1,\tfrac{H}{n}\big\} \Bigl)^{n} d\sigma_g.
        \end{equation}
 The equality holds if and only if $\partial\Omega$ is connected, $H\equiv n$ and $\overline{M\setminus\Omega}$ is isometric to
$$
\bigl([0,\infty)\times\partial \Omega,\; dr^2+e^{2r}g_{\partial\Omega}\bigr)$$
where $g_{\partial\Omega}$ is the induced metric on $\partial \Omega$.
   \item  [(2)] If the mean curvature $H>n,$ then
   \begin{equation}\label{eq1.2}
   \mathrm{Cap}_p(\Omega)\leq\frac{\int_{\partial\Omega}\big(\tfrac{H^2}{n^2}-1\big)^{\frac n2}\cdot v_p(\mathrm{arccoth}\tfrac Hn)d\sigma_g}
   {\big(\min\limits_{x\in\partial\Omega}v_p(\mathrm{arccoth}\tfrac {H(x)}n)\big)^p}.
   \end{equation}
   The equality holds if and only if $\partial\Omega$ is connected, $H\equiv n\coth \theta_0>n$ is a constant for some $\theta_0>0$ and $\overline{M\setminus\Omega}$ is isometric to
$$
\Bigl([0,\infty)\times\partial \Omega,\; dr^2+\Big(\frac{\sinh (r+\theta_0)}{\sinh \theta_0}\Big)^2 g_{\partial\Omega}\Bigr).$$

   \item [(3)] If the mean curvature $H\in[0,n),$ then
   \begin{equation}\label{eq1.3}
   \mathrm{Cap}_p(\Omega)\leq\frac{\int_{\partial \Omega}\big(1-\tfrac{H^2}{n^2}\big)^{\frac n2}\cdot w_p(\mathrm{arctanh}\tfrac Hn)d\sigma_g}{\big(\min\limits_{x\in\partial\Omega}w_p(\mathrm{arctanh}\tfrac {H(x)}n)\big)^p}.
   \end{equation}
   The equality holds if and only if $H\equiv n\tanh \delta_0\in[0,n)$ is a constant for some $\delta_0\geq0$ and $\overline{M\setminus\Omega}$ is isometric to
$$
\Bigl([0,\infty)\times\partial \Omega,\; dr^2+\Big(\frac{\cosh (r+\delta_0)}{\cosh \delta_0}\Big)^2 g_{\partial\Omega}\Bigr).$$
 \end{itemize}
\end{theorem}
\begin{remark}
\begin{itemize}
  \item  [(1)] In general, $v_p$ and $w_p$ do not admit simple closed-form expressions. However, for certain special values of $p,$ they can be expressed explicitly in terms of elementary functions. For instance, when $p=n+1,$
      $$
      v_{n+1}(\mathrm{arccoth}\tfrac Hn)=\mathrm{arccosh}(\tfrac Hn),\ \ \
       w_{n+1}(\mathrm{arctanh}\tfrac Hn)=\mathrm{arccos}(\tfrac Hn).
      $$
      When $p=\frac n2+1,$ then
      $$
      v_{\frac n2+1}(\mathrm{arccoth}\tfrac Hn)=\frac Hn-1,\ \ \
       w_{\frac n2+1}(\mathrm{arctanh}\tfrac Hn)=1-\frac Hn.
      $$
  \item  [(2)] If the equalities hold in Theorem \ref{mainthm}, then
  $$
      \frac{\mathrm{Cap}_p(\Omega)}{|\partial \Omega|}=\big(\frac{n}{p-1}\big)^{p-1}=\frac{\mathrm{Cap}_p(B_1(R),M_1)}{|\partial B_1(R)|_{g_1}}\ \   {\rm in}\ \eqref{eq1.1}$$
    $$
     \frac{\mathrm{Cap}_p(\Omega)}{|\partial \Omega|}=\frac{ \mathrm{Cap}_p(B_2(\theta_0),M_2)}{|\partial B_2(\theta_0)|_{g_2}} \ \ {\rm in}\ \eqref{eq1.2}
     $$
     $$
      \frac{\mathrm{Cap}_p(\Omega)}{|\partial \Omega|}=\frac{ \mathrm{Cap}_p(B_3(\delta_0),M_3)}{|\partial B_3(\delta_0)|_{g_3}}\ \  {\rm in}\ \eqref{eq1.3}.
  $$
  respectively.
  \item  [(3)] If the equality in \eqref{eq1.3} holds, then $\partial\Omega$ may not be connected, although the mean curvature $H$ is constant (and takes the same value) on every connected component. For example, $\Omega=B_3(R)\subseteq M_3.$ Then $\partial\Omega$ consists of two disjoint spheres $\{r=R\}$ and $\{r=-R\},$ both with mean curvature $H=n\tanh R\in[0,n).$
\end{itemize}
\end{remark}
We now consider the case ${\rm Ric}\geq 0.$ The volume comparison becomes polynomial, yielding a different sharp upper bound for the
$p$-capacity.
\begin{theorem}
  Assume that $(M,g)$ is a complete non-compact Riemannian manifold of dimension $n+1$ with $\mathrm{Ric}\ge 0$ and $p\in(1,n+1).$ Let $\Omega\subset M$ be a compact smooth domain and $H>0$ be the mean curvature of $\partial\Omega.$ Then
  \begin{equation}\label{eq1.4}
          \mathrm{Cap}_p(\Omega)\leq\Big(\frac{n+1-p}{(p-1)n}\Big)^{p-1}\big(\max\limits_{x\in\partial\Omega} H(x)\big)^p \int_{\partial \Omega} \frac{1}{H}\, d\sigma_g.
  \end{equation}
  The equality holds if and only if $\partial\Omega$ is connected, $H\equiv h_0$ is a constant and $\overline{M\setminus\Omega}$ is isometric to
  $$
 \Big([0,\infty)\times\partial\Omega,\ dr^2+\big(1+\frac{h_0}{n}r\big)^2g_{\partial\Omega}\Big).
 $$
\end{theorem}
\vspace{0.5cm}
In the second part of this paper, we turn to the study of lower bounds for the $p$-capacity under the condition that the Ricci curvature is bounded from below.
In general, the absolute capacity $\mathrm{Cap}_p(\Omega)$ may vanish; for instance, on a cylinder or in $\mathbb{R}^n$ with $p\geq n$ the capacity of a compact set can be zero. Therefore it is natural to consider the relative $p$-capacity
${\rm Cap}_p(K,O).$
\par Let $(M,g)$ be a Riemannian manifold of dimension $n+1$. A pair $(K,O)$ is called a condenser in $M$ if $K$ is a compact subset of a bounded open domain $O\subseteq M$. For any $p>1$, the relative $p$-capacity of $(K,O)$ is defined by
$$
{\rm Cap}_p(K,O)=\inf\left\{\int_{O}|\nabla u|^p\,dv_g:\ 1_K\leq u\in C^\infty_0(O)\right\}.
$$
For any $\kappa\geq 0,$ define the manifold space:
$$\mathcal{M}_\kappa=\Big\{(M,g): M \text{ is a complete non-compact manifold of dimension } n+1\ \text{and}\ \operatorname{Ric}_g\geq -n\kappa g\Big\}.$$
and for any $\lambda>0,$ we consider the dilation:
$$
S_\lambda \colon \mathcal{M}_\kappa \to \mathcal{M}_{\lambda^{-2} \kappa},\ \ (M,g) \mapsto (M,\tilde g = \lambda^2 g).
$$
Under the dilation $g\mapsto\lambda^2g$, the condition $\operatorname{Ric}_g\geq-n\kappa g$ becomes $\operatorname{Ric}_{\widetilde g}\geq-n\lambda^{-2}\kappa\widetilde g$. Thus $S_\lambda$ is one-to-one. Moreover, the relative capacity scales as
$$
{\rm Cap}_p(K,O)_{\lambda^2 g}=\lambda^{n+1-p}{\rm Cap}_p(K,O)_g.
$$
To obtain a scale-invariant quantity, we normalize $\mathrm{Cap}_p(K,O)$ by suitable geometric factors of
$K$ and $O.$ For any $\mu\in\mathbb{R},$ if  $\kappa\geq0,$ then  we define
$$
  \mathcal{C}^{(M,g)}_{p,\mu}(K,O) = \frac{\mathrm{Cap}_p(K,O)}{\mathrm{diam}(O)^{(n+1)(1-\mu)-p}|K|^\mu},$$
A direct computation shows that the exponents are chosen so that $\mathcal{C}^{(M,g)}_{p,\mu}(K,O)$ is scale-invariant for any $\mu$.
On the other hand, if $\kappa>0,$ then we consider
  $$\mathcal{D}_{p,\mu}^{(M,g)}(K,O) = \kappa^{\frac{n+1-p-(n+1)\mu}{2}} \mathrm{Cap}_p(K,O) \left(\frac{e^{n\sqrt{\kappa}\,\mathrm{diam}(O)}}{|K|}\right)^\mu.$$
Notice that
$$\mathrm{diam}_{\tilde g}(O) = \lambda \mathrm{diam}_g(O),\ \ |K|_{\tilde g} = \lambda^{n+1}|K|_g.$$ Therefore
$$
  \mathcal{D}^{(M,g)}_{p,\mu}(K,O) = \mathcal{D}^{(M,\tilde{g})}_{p,\mu}(K,O).
$$
\\~
We establish lower bounds for $\mathcal{C}_{p,\mu}(K,O)$ and $\mathcal{D}_{p,\mu}(K,O)$ on $\mathcal{M}_\kappa$.
\begin{theorem}\label{thm:1.6}
  Assume that $p>1$, $\mu\in\mathbb{R}$, and $\kappa\geq0$, and let $\mathcal{M}_\kappa$, $\mathcal{C}^{(M,g)}_{p,\mu}(K,O)$, and $\mathcal{D}^{(M,g)}_{p,\mu}(K,O)$ be defined as above. Then
  \begin{enumerate}
    \item For $\kappa= 0,$ \begin{equation}\label{eq1.5}
  \inf_{(M,g) \in \mathcal{M}_0}\ \inf_{K\subseteq O\subseteq M}\mathcal{C}^{(M,g)}_{p,\mu}(K,O) > 0\ \ \text{if and only if}\ \ \mu \geq 1.
  \end{equation}

  \item For all $\kappa> 0$, \begin{equation}\label{eq1.6}
  \inf_{(M,g) \in \mathcal{M}_\kappa}\ \inf_{K\subseteq O\subseteq M}\mathcal{C}^{(M,g)}_{p,\mu}(K,O) = 0.
  \end{equation}

  \item For all $\kappa > 0$ and $D_0 > 0$,
  \begin{equation}\label{eq:1.7}
    \inf_{(M,g) \in \mathcal{M}_\kappa} \ \inf_{\substack{K \subseteq O \subseteq M \\ \mathrm{diam}(O) \leq D_0}} \mathcal{C}^{(M,g)}_{p,\mu}(K,O) > 0 \text{ if and only if } \mu \geq 1.
  \end{equation}

  \item For all $\kappa>0$,
  \begin{equation}\label{eq:1.8}
    \inf_{(M,g) \in \mathcal{M}_\kappa} \ \inf_{K \subseteq O \subseteq M} \mathcal{D}^{(M,g)}_{p,\mu}(K,O) > 0 \text{ if and only if } \mu \geq 1.
  \end{equation}

  \end{enumerate}
\end{theorem}

\begin{remark}\label{rmk1.7}
  One may also consider using other geometric factors of $K$ and $O$ to normalize $\mathrm{Cap}_p(K,O)$, for example, $|K|,|O|,|\partial K|,|\partial O|$. Consider
  $$
    \mathcal{C}^{(M,g)}_{p,a,b,c,d}(K,O) = \frac{\mathrm{Cap}_p(K,O)}{|K|^a|O|^b|\partial K|^c|\partial O|^d},
  $$
  with $(n+1)(a+b)+n(c+d)=n+1-p$ to ensure scale-invariance. If  $b > 1-p$ or $a+b \neq 1-p,$ then

  $$
  \inf_{(M,g) \in \mathcal{M}_0}\ \inf_{K\subseteq O\subseteq M}\mathcal{C}^{(M,g)}_{p,a,b,c,d}(K,O) = 0.$$
See the proof at the end of Section 4.
\end{remark}

The paper is organized as follows. In Section 2, we recall the required comparison results for the normal Jacobian and prove the sharp upper bounds for the absolute $p$-capacity. Section 3 contains two applications: relative-capacity comparison inequalities for condensers and capacity--volume estimates on rotationally symmetric manifolds. In Section 4, we prove Theorem \ref{thm:1.6}, determine the admissible ranges of the normalization parameter, and derive a sharp Euclidean lower bound for $\mathcal{C}_{p,1}^{(M,g)}(K,O)$.

\section{Proof for the upper bound of capacity}
\subsection{Preliminaries}
Let $(M^{n+1},g)$ be a complete Riemannian manifold with Ricci curvature satisfying $\mathrm{Ric}\ge -ng$. Let $\Omega\subset M$ be a bounded smooth domain and denote by $r(\cdot)=d_g(\cdot,\partial\Omega)$ the distance to the boundary. $r$ is smooth almost everywhere in $M\setminus\Omega$ and $|\nabla r|\equiv1$. Along the normal geodesics $\sigma_x(t)$ emanating from $x\in\partial\Omega$, we define
$$
\tau(x)=\sup\{t>0:\ r(\sigma_x(t))=t\},\ \ \partial\Omega\rightarrow \mathbb{R}\cup\{\infty\}
$$
and then the cut
locus for $\partial\Omega$ in $M\setminus\Omega$ is defined by
$${\rm Cut}(\partial\Omega)=\{\sigma_x(\tau(x)):x\in \partial\Omega\}.$$
Let $E=\{(x,r)\in\partial\Omega\times[0,+\infty):r<(\tau(x))\}$ and we have the diffeomorphism
$$
\Phi:E\rightarrow (\overline{M\setminus\Omega})\setminus {\rm Cut}(\partial\Omega),\ \ \ (x,r)\mapsto \sigma_x(r).
$$
 Then the volume element can be written as
$$dv_g=J(x,r)\,dr\,d\sigma(x),\ \ J(x,0)=1.$$
We can extend the definition of $J(x,r)$ to $\partial\Omega\times[0,+\infty)$ by requiring that
$$
J(x,r)=0,\ \ {\rm if}\ r\geq(\tau(x)).
$$
According to the Riccati equation and the comparison theory (see Lemma 2.1 in \cite{jin2024willmore}), we have
\begin{equation}\label{eq2.1}
  J(x,r)\le \Bigl(\cosh r+\tfrac{H(x)}{n}\sinh r\Bigl)^{n},\quad \forall x\in\partial\Omega\ \ \&\ \ r\ge0.
\end{equation}
\subsection{Proof of Theorem 1.3}
 We are about to prove the main Theorem one-by-one.
\begin{itemize}
  \item  Notice that \eqref{eq2.1} implies that
$$
J(x,r)\le \Bigl(\cosh r+\tfrac{H(x)}{n}\sinh r\Bigl)^{n}\le\bigl(\max\{1,\tfrac{H(x)}{n}\}\bigr)^{n} e^{nr},\ \  \forall r\ge0
$$ and when $r>0,$ then the second equality holds if and only if $H(x)=n.$ Choose the test function
$\varphi(\cdot)=e^{\frac{nr(\cdot)}{1-p}}$
on $M\setminus\Omega,$ then $\varphi$ is Lipschitz continuous and
$$
\varphi|_{\partial\Omega}=1,\ \ \&\ \ \lim\limits_{r\rightarrow\infty}\varphi=0.
$$
Hence
\begin{equation}\label{eq2.2}
\begin{aligned}
  \mathrm{Cap}_p(\Omega)&\leq\int_{M\setminus\Omega}|\nabla \varphi|^p\,dv_g
  \\ & = \int_{\partial\Omega}\int_0^\infty \Big|\frac{n}{1-p}\Big|^p e^{\frac{np}{1-p}r}J(x,r)\,dr\,d\sigma_g(x)
  \\ & \leq \Big(\frac{n}{p-1}\Big)^p\int_{\partial\Omega}\bigl(\max\{1,\tfrac{H(x)}{n}\}\bigr)^{n} \int_0^\infty e^{\frac{np}{1-p}r} e^{nr}\,dr\,d\sigma_g(x)
  \\&=\Big(\frac{n}{p-1}\Big)^{p-1}\int_{\partial\Omega}\bigl(\max\{1,\tfrac{H(x)}{n}\}\bigr)^{n} d\sigma_g(x).
 \end{aligned}
\end{equation}
 From the proof we know that if the equality holds, then $\varphi$ is the $p$-potential of $\Omega$ and
$$
J(x,r)= \bigl(\max\{1,\tfrac{H(x)}{n}\}\bigr)^{n} e^{nr}
$$ almost everywhere. Then the mean curvature and the volume form satisfies
$$
\begin{cases}
  H(x)=n,\ \ \forall x\in\partial\Omega, \\
  J(x,r)=e^{nr}\ \ \forall r\geq 0.
\end{cases}$$
Then ${\rm Cut}(\partial\Omega)$ is empty and on $\Phi([0,\infty)\times\partial\Omega)$ the second fundamental form of each $r-$level set satisfies that
$$
D^2r=\frac{H(x,r)}{n}g=\frac{J'(x,r)}{nJ(x,r)}g=g
$$ Locally, $g=dr^2+g_r,\ \ r\geq0$ on $M\setminus\Omega$ and $\frac12\frac{\partial}{\partial r} g_{ij}=g_{ij}$ and hence
$$g=dr^2+e^{2r}g|_{\partial\Omega}.$$
To check the connectness of $\partial\Omega,$ we only need to show that $M$ has one end. Let us consider one connected component $\Sigma$ of $\partial\Omega$ and a family of compact hypersurfaces $\{\Sigma_k=\{k\}\times\Sigma\}$ for $k=1,2,\cdots.$ Notice that $d_g(\Sigma,\Sigma_k)\rightarrow\infty$ and the mean curvature $\Sigma_k$ is $H|_{\Sigma_k}\equiv n.$ According to Theorem 3 in \cite{cai2000boundaries} we know that either $M$ has one end or $(M,g)$ isometric to $(\mathbb{R}\times\Sigma,  dr^2 + e^{2r}g_\Sigma).$
If the second case happens, since $H|_{\partial\Omega}\equiv n,$ then there must be
$$\Omega=(-\infty,0]\times\Sigma$$
 and contradicts the fact that $\Omega$ is a compact set. As a consequence, $\partial\Omega$ is connected.
\par Lastly, it is easy to check that $\varphi$ is the $p$-potential of $\Omega$ in $M,$ so the equality in \eqref{eq1.1} could be achieved and the estimate is sharp.
  \item If $H(x)>n$ for all $x\in\partial\Omega,$ then by \eqref{eq2.1},
  \begin{equation}\label{eq2.3}
    J(x,r)\le \big(\cosh r+\tfrac{H(x)}{n}\sinh r\big)^{n}=\big(\tfrac{H(x)^2}{n^2}-1\big)^{\frac n2}\big(\sinh(r+\theta_x)\big)^n
  \end{equation}
  where $\coth\theta_x=\frac{H(x)}{n}>1.$ Define $\theta_0=\max\limits_{x\in\partial\Omega}\theta_x$ and consider the test function
  $$
  \varphi(\cdot)=\frac{\int_{r(\cdot)+\theta_0}^\infty\big(\sinh t\big)^{\frac n{1-p}}\,dt}{\int_{\theta_0}^\infty\big(\sinh t\big)^{\frac n{1-p}}\,dt}
  =\frac{v_p(r(\cdot)+\theta_0)}{v_p(\theta_0)}
  \ \ \text{with}\ \ \varphi|_{\partial\Omega}=1\ \ \&\ \ \varphi|_\infty=0.
  $$
  Then
  \begin{equation}\label{eq2.4}
\begin{aligned}
  \mathrm{Cap}_p(\Omega)&\leq\int_{M\setminus\Omega}|\nabla \varphi|^p\,dv_g
  \\ & = \int_{\partial\Omega}\int_0^\infty \frac{1}{\big(v_p(\theta_0)\big)^p}\big(\sinh (r+\theta_0)\big)^{\frac {np}{1-p}}J(x,r)\,dr\,d\sigma_g(x)
  \\ & \leq  \frac{1}{\big(v_p(\theta_0)\big)^p}\int_{\partial\Omega}\big(\tfrac{H(x)^2}{n^2}-1\big)^{\frac n2} \int_0^\infty
  \big(\sinh (r+\theta_0)\big)^{\frac {np}{1-p}}\big(\sinh(r+\theta_x)\big)^n \,dr\,d\sigma_g(x)
  \\&\leq  \frac{1}{\big(v_p(\theta_0)\big)^p}\int_{\partial\Omega}\big(\tfrac{H(x)^2}{n^2}-1\big)^{\frac n2} \int_0^\infty
  \big(\sinh (r+\theta_x)\big)^{\frac {n}{1-p}} \,dr\,d\sigma_g(x)
  \\&= \frac{1}{\big(v_p(\theta_0)\big)^p}\int_{\partial\Omega}\big(\tfrac{H(x)^2}{n^2}-1\big)^{\frac n2} v_p(\theta_x)\,d\sigma_g(x)
 \end{aligned}
\end{equation}
 and \eqref{eq1.2} is proved. If the equality holds, then $\varphi$ is the $p$-potential,
  $$
  J(x,r)= \big(\cosh r+\tfrac{H(x)}{n}\sinh r\big)^{n}\ \ \&\ \ \theta_x=\theta_0
  $$
  almost everywhere. Then the mean curvature $H(x)=n\coth \theta_0>n$ is constant and
  $$
 J(x,r)=\big(\sinh\theta_0)^{-n}\big(\sinh(r+\theta_0)\big)^n,\ \ \forall x\in\partial\Omega\ \ \&\ \ r\ge0.
  $$
  Then ${\rm Cut}(\partial\Omega)$ is empty and on $\Phi([0,\infty)\times\partial\Omega)$ the second fundamental form of each $r-$level set satisfies that
$$
D^2r=\frac{H(x,r)}{n}g=\frac{J'(x,r)}{nJ(x,r)}g=\coth(r+\theta_0)g.$$
Hence we can conclude that in $\Phi([0,\infty)\times\partial\Omega)=\overline{M\setminus\Omega},$
$$
g=dr^2+\Big(\frac{\sinh (r+\theta_0)}{\sinh \theta_0}\Big)^2 g_{\partial\Omega}.
$$
Similar to case 1, we can also use Theorem 3 in \cite{cai2000boundaries} to prove that $M$ has only one end and hence $\partial\Omega$ is connected.
\par In the end,  $\varphi$ is the $p$-potential of $\Omega$ in $M,$ so the equality in \eqref{eq1.2} could be achieved.
  \item If $H(x)\in[0,n)$ for all $x\in\partial\Omega,$ then \eqref{eq2.1} implies that
  \begin{equation}\label{eq2.5}
    J(x,r)\le \big(\cosh r+\tfrac{H(x)}{n}\sinh r\big)^{n}=\big(1-\tfrac{H(x)^2}{n^2}\big)^{\frac n2}\big(\cosh(r+\delta_x)\big)^n
  \end{equation}
  where $\tanh\delta_x=\frac{H(x)}{n}\in[0,1).$ Define $\delta_0=\max\limits_{x\in\partial\Omega}\delta_x$ and consider the test function
  $$
  \varphi(\cdot)=\frac{\int_{r(\cdot)+\delta_0}^\infty\big(\cosh t\big)^{\frac n{1-p}}\,dt}{\int_{\delta_0}^\infty\big(\cosh t\big)^{\frac n{1-p}}\,dt}
  =\frac{w_p(r(\cdot)+\delta_0)}{w_p(\delta_0)}
  \ \ \text{with}\ \ \varphi|_{\partial\Omega}=1\ \ \&\ \ \varphi|_\infty=0.
  $$ Then
  \begin{equation}\label{eq2.6}
\begin{aligned}
  \mathrm{Cap}_p(\Omega)&\leq\int_{M\setminus\Omega}|\nabla \varphi|^p\,dv_g
  \\ & = \int_{\partial\Omega}\int_0^\infty \frac{1}{\big(w_p(\delta_0)\big)^p}\big(\cosh (r+\delta_0)\big)^{\frac {np}{1-p}}J(x,r)\,dr\,d\sigma_g(x)
  \\ & \leq  \frac{1}{\big(w_p(\delta_0)\big)^p}\int_{\partial\Omega}\big(1-\tfrac{H(x)^2}{n^2}\big)^{\frac n2} \int_0^\infty
  \big(\cosh (r+\delta_0)\big)^{\frac {np}{1-p}}\big(\cosh(r+\delta_x)\big)^n \,dr\,d\sigma_g(x)
  \\&\leq  \frac{1}{\big(w_p(\delta_0)\big)^p}\int_{\partial\Omega}\big(1-\tfrac{H(x)^2}{n^2}\big)^{\frac n2} \int_0^\infty
  \big(\cosh (r+\delta_x)\big)^{\frac {n}{1-p}} \,dr\,d\sigma_g(x)
  \\&= \frac{1}{\big(w_p(\delta_0)\big)^p}\int_{\partial\Omega}\big(1-\tfrac{H(x)^2}{n^2}\big)^{\frac n2} w_p(\delta_x)\,d\sigma_g(x)
 \end{aligned}
\end{equation}
The proof of the equality is the same as the second case, except that $\partial\Omega$ may not be connected.
\end{itemize}
\subsection{Proof of Theorem 1.5}
\begin{proof}[Proof of Theorem 1.5]
The proof is similar to  Theorem 1.3 and much easier. We will continue using the notation from Section 2.1 and notice that
if ${\rm Ric}\geq 0,$ then by
the Riccati equation and the comparison theory, we have
\begin{equation}\label{eq2.7}
  J(x,r)\le \Bigl(1+\tfrac{H(x)}{n} r\Bigl)^{n},\quad \forall x\in\partial\Omega\ \ \&\ \ r\ge0.
\end{equation}
Consider the test function
 $$\varphi(\cdot)=\big(1+\tfrac{h_0}{n}r(\cdot)\big)^{\frac{n+1-p}{1-p}}\ \ {\rm with}\ \ h_0=\max\limits_{x\in\partial\Omega} H(x).$$
 If $p\in (1,n+1),$ then $\lim\limits_{r\rightarrow\infty}\varphi=0$ and hence
  \begin{equation}\label{eq2.8}
\begin{aligned}
  \mathrm{Cap}_p(\Omega)&\leq\int_{M\setminus\Omega}|\nabla \varphi|^p\,dv_g
  \\ &=\int_{\partial\Omega}\int_0^\infty\Big(\frac{n+1-p}{p-1}\Big)^p\Big(\frac{h_0}{n}\Big)^p\Big(1+\tfrac{h_0}{n}r\Big)^{\frac{np}{1-p}}
  J(x,r)\,dr\,d\sigma_g(x)
   \\ & \leq \Big(\frac{n+1-p}{p-1}\Big)^p\Big(\frac{h_0}{n}\Big)^p\int_{\partial\Omega}\int_0^\infty\Big(1+\frac{H(x)}{n}r\Big)^{\frac{np}{1-p}}
   \Bigl(1+\tfrac{H(x)}{n} r\Bigl)^{n}\,dr\,d\sigma_g(x)
   \\&=\Big(\frac{n+1-p}{p-1}\Big)^{p-1}\Big(\frac{h_0}{n}\Big)^p\int_{\partial\Omega}\frac{n}{H(x)}\,d\sigma_g(x)
 \end{aligned}
 \end{equation}
 If the equality holds, then $H(x)=h_0$ is a constant on $\partial\Omega$ and
 $$J(x,r)=\Bigl(1+\frac{h_0}{n} r\Bigl)^{n},\quad \forall x\in\partial\Omega\ \ \&\ \ r\ge0.
 $$
 Hence
 $$
 \frac{1}{2}\partial_r g=D^2 r=\frac{H(x,r)}{n}g=\frac{J'(x,r)}{nJ(x,r)}g=\frac{1}{\frac{n}{h_0}+r}g
 $$
 which would imply that $(\overline{M\setminus\Omega},g)$ is isometric to
 $$
 \Big([0,\infty)\times\partial\Omega,\ dr^2+\big(1+\frac{h_0}{n}r\big)^2g_{\partial\Omega}\Big).
 $$
 In the end, $\partial\Omega$ is connected  due to Theorem B in \cite{kasue1983ricci} and Cheeger-Gromoll splitting theorem.
\end{proof}

\section{Some applications}

\subsection{The estimates of relative capacity}

Let us consider the relative $p$-capacity $\mathrm{Cap}_p(K,O)$, where $(K,O)$ is a condenser in $M$. There is also a $p$-capacitary potential theory for the relative capacity. Similar to Theorem \ref{mainthm}, we also have the estimates of relative capacity:
\begin{theorem}\label{thm3.1}
 Assume that $p>1$ and $(M,g)$ is a complete Riemannian manifold of dimension $n+1$ with $\mathrm{Ric}\ge -ng.$ Let $(K,O)$ be a condenser with smooth boundaries and $H$ be the mean curvature of $\partial K.$  If the distance between $\partial K$ and $\partial O$ is $R,$ then we have the following three sharp estimates:
 \begin{itemize}
   \item [(1)]
   \begin{equation}\label{eq3.1}
          \mathrm{Cap}_p(K,O)\le \Big(\frac{n}{p-1}\Big)^{p-1}\Big(1-e^{\frac{nR}{1-p}}\Big)^{1-p}  \int_{\partial K} \Bigl( \max\big\{1,\tfrac{H}{n}\big\} \Bigl)^{n} d\sigma_g.
        \end{equation}
 The equality holds if and only if $H\equiv n$ and $\overline{O\setminus K}$ is isometric to
$$
\bigl([0,R]\times\partial K,\; dr^2+e^{2r}g_{\partial K}\bigr)$$
   \item  [(2)] If the mean curvature $H>n,$ then
   \begin{equation}\label{eq3.2}
   \mathrm{Cap}_p(K,O)\leq\frac{\int_{\partial K}\big(\tfrac{H^2}{n^2}-1\big)^{\frac n2}\cdot \big(v_p(\theta_x)-v_p(R+\theta_x)\big)d\sigma_g}
   {\big(v_p(\theta_0)-v_p(R+\theta_0)\big)^p}.
   \end{equation}
   where $\theta_x={\rm arccoth}\frac{H(x)}{n}$ and $\theta_0=\max\limits_{x\in\partial K}\theta_x.$
   The equality holds if and only if $H\equiv n\coth \theta_0>n$ is constant for some $\theta_0>0$ and $\overline{O\setminus K}$ is isometric to
$$
\bigl([0,R]\times\partial K,\; dr^2+\Big(\frac{\sinh (r+\theta_0)}{\sinh \theta_0}\Big)^2 g_{\partial K}\bigr)$$

   \item [(3)] If the mean curvature $H\in[0,n),$ then
   \begin{equation}\label{eq3.3}
   \mathrm{Cap}_p(K,O)\leq\frac{\int_{\partial K}\big(1-\tfrac{H^2}{n^2}\big)^{\frac n2}\cdot \big(w_p(\delta_x)-w_p(R+\delta_x)\big)d\sigma_g}
   {\big(w_p(\delta_0)-w_p(R+\delta_0)\big)^p}
   \end{equation}
    where $\delta_x={\rm arctanh}\frac{H(x)}{n}$ and $\delta_0=\max\limits_{x\in\partial K}\delta_x.$
   The equality holds if and only if $H\equiv n\tanh \delta_0\in[0,n)$ is constant for some $\delta_0\geq0$ and $\overline{O\setminus K}$ is isometric to
$$
\bigl([0,R]\times\partial K,\; dr^2+\Big(\frac{\cosh (r+\delta_0)}{\cosh \delta_0}\Big)^2 g_{\partial K}\bigr)$$
 \end{itemize}
\end{theorem}
\begin{proof}
The proof is nearly the same as that of Theorem \ref{mainthm}. Notice that if
$$
K_R=\{x\in O\:\ d_g(x, K)< R\},
$$
then
$$
K_R\subset O\ \ \&\ \ {\rm Cap}_p(K,O)\leq {\rm Cap}_p(K,K_R)
$$
For the first case, define
$$
\varphi(\cdot)=\frac{e^{\frac{nr(\cdot)}{1-p}}-e^{\frac{nR}{1-p}}}{1-e^{\frac{nR}{1-p}}}\ \ {\rm on}\ \ K_R\setminus K
$$ to be the test function satisfying $\varphi|_{\partial K}=1\ \&\ \varphi|_{\partial K_R}=0,$ then
\begin{equation}\label{eq3.4}
\begin{aligned}
{\rm Cap}_p(K,K_R) &\leq\int_{K_R\setminus K}|\nabla \varphi|^p\,dv_g
  \\ & = \Big|\frac{n}{1-p}\Big|^p \Big(1-e^{\frac{nR}{1-p}}\Big)^{-p} \int_{\partial K}\int_0^R e^{\frac{np}{1-p}r}J(x,r)\,dr\,d\sigma_g(x)
  \\ & \leq \Big(\frac{n}{p-1}\Big)^p\Big(1-e^{\frac{nR}{1-p}}\Big)^{-p}  \int_{\partial K}\bigl(\max\{1,\tfrac{H(x)}{n}\}\bigr)^{n} \int_0^R e^{\frac{np}{1-p}r} e^{nr}\,dr\,d\sigma_g(x)
  \\&=\Big(\frac{n}{p-1}\Big)^{p-1}\Big(1-e^{\frac{nR}{1-p}}\Big)^{1-p}\int_{\partial K}\bigl(\max\{1,\tfrac{H(x)}{n}\}\bigr)^{n} d\sigma_g(x).
\end{aligned}
\end{equation}
When the equality in \eqref{eq3.1} holds, on the one hand, by the same analysis in Section 3, we get that
$H\equiv n,$ $(\overline{K_R\setminus K},g)$ is isometric to
$$
\bigl([0,R]\times\partial K,\; dr^2+e^{2r}g_{\partial K}\bigr)$$
and $\varphi$ is the $p$-potential of $(K,K_R).$  On the other hand, let $\bar{\varphi}$ denote the extension of $\varphi$ by zero to
$O\setminus K$ as long as $O\setminus K$ is nonempty. Then $\bar{\varphi}$ is Lipschitz continuous and
 $${\rm Cap}_p(K,O)= {\rm Cap}_p(K,K_R)=\int_{K_R\setminus K}|\nabla \varphi|^p\,dv_g=\int_{O\setminus K}|\nabla \bar{\varphi}|^p\,dv_g.$$
 By the uniqueness of the $p$-capacitary potential, $\bar{\varphi}$ is $p$-harmonic in $O\setminus K.$
 The normal derivative of $\bar{\varphi}$ across $\partial K_R$ would then jump from a negative value to zero. This discontinuity would imply that
$\varphi$ is not $p$-harmonic in a neighborhood of $\partial K_R$ as the distributional $p$-Laplacian would have a singular measure supported on
 $\partial K_R$, contradicting the fact that it is $p$-harmonic in $O\setminus K.$
 Therefore $O=K_R$ and the rigidity result holds.
 \par For the second case $(H>n),$ we choose the test function
 $$
 \varphi(\cdot)=\frac{\int_{r(\cdot)+\theta_0}^{R+\theta_0}(\sinh t)^{-\frac{n}{p-1}}\,dt}{\int_{\theta_0}^{R+\theta_0}(\sinh t)^{-\frac{n}{p-1}}\,dt}
=\frac{v_p(r(\cdot)+\theta_0)-v_p(R+\theta_0)}{v_p(\theta_0)-v_p(R+\theta_0)}
 $$
 and for the third case $(H\in[0,n),$ we choose
  $$
 \varphi(\cdot)=\frac{\int_{r(\cdot)+\delta_0}^{R+\delta_0}(\cosh t)^{-\frac{n}{p-1}}\,dt}{\int_{\delta_0}^{R+\delta_0}(\cosh t)^{-\frac{n}{p-1}}\,dt}
=\frac{w_p(r(\cdot)+\delta_0)-w_p(R+\delta_0)}{w_p(\delta_0)-w_p(R+\delta_0)}.
 $$
The proofs are similar to that of Theorem \ref{mainthm} and are omitted here.
\end{proof}
In the equality cases, $\partial K$ may have several connected components; the rigidity then holds for each component separately, with the same constant mean curvature on all components. This does not affect the validity of the isometric characterization, as the warped product structure applies to each component individually.

\subsection{A volume upper bound via mean curvature}

In this section we combine the lower bound for the $p$-capacity recently obtained by Jin and Xiao \cite{jin2025essential} with our upper bound to derive an explicit estimate for the volume of a domain in terms of its boundary mean curvature.

Let
$(M^{n+1},g)$ be a complete non-compact rotationally symmetric manifold with $\mathrm{Ric}\leq 0$ satisfying the isoperimetry of $o = \{r = 0\}-$centered balls - namely - for any compact subset $K$ in $M,$
$$
|K|=|B(o,R)|\Rightarrow|\partial K|\geq|\partial B(o,R)|.
$$
For a smooth condenser $(K,O)$ in $M$, it is proved in \cite{jin2025essential} that the following three sharp lower bound for the $p$-capacity holds:
$$\begin{cases}
  {\rm Cap}_p(K,O)\geq (n+1) v_{n+1}^{\frac{p}{n+1}}\Big(\frac{n+1-p}{p-1}\Big)^{p-1}\,|K|^{1-\frac{p}{n+1}}\ & {\rm if} \ p\in(1,n+1); \\
   \exp\left(-(n+1)^{\frac{n+1}{n}}v_{n+1}^{\frac1n} \Big({\rm Cap}_{n+1}(K,O)\Big)^{-\frac1n}\right)\geq\frac{|K|}{|O|}  & {\rm if} \ p=n+1; \\
   {\rm Cap}_p(K,O)\geq  (n+1) v_{n+1}^{\frac{p}{n+1}}\Big(\frac{p-n-1}{p-1}\Big)^{p-1} |O|^{1-\frac{p}{n+1}}  & {\rm if} \ p>n+1.
  \end{cases}
$$
where $|K|$ and $|O|$ denote the Riemannian volume and $v_{n+1}$ denote the volume of the unit open ball $\mathbb{B}^{n+1}$ in $\mathbb{R}^{n+1}.$
Combining with Theorem \ref{thm3.1}, we obtain
\begin{theorem}
Let
$$(M^{n+1},g)=([0,\infty)\times\mathbb{S}^n,dr^2+\varphi(r)^2g_{\mathbb{S}^n})$$
$(M^{n+1},g)$ be a rotationally symmetric manifold with $-ng\leq\mathrm{Ric}\leq 0$ satisfying the isoperimetry of $o = \{r = 0\}-$centered balls, and $(K,O)$ be a smooth condenser in $M$. If $H$ is the mean curvature of $\partial K$ and $R=d_g(\partial K,\partial O),$ then for every $p>1,$
$$\begin{cases}
    |K|^{1-\frac{p}{n+1}}\leq\frac{(\frac{n}{n+1-p})^{p-1}\Big(1-e^{\frac{nR}{1-p}}\Big)^{1-p}}{(n+1)v_{n+1}^{\frac{p}{n+1}}}\int_{\partial K} \Bigl( \max\big\{1,\tfrac{H}{n}\big\} \Bigl)^{n} d\sigma_g \ & {\rm if}\ p\in(1,n+1); \\
   \frac{(n+1)^{\frac{n+1}{n}}v_{n+1}^{\frac1n} }{\ln|O|-\ln|K|}\leq\frac{\Big(\int_{\partial K} \bigl(\max\big\{1,\tfrac{H}{n}\big\} \bigr)^{n} d\sigma_g\Big)^{\frac1n}}{1-e^{-R}} & {\rm if} \ p=n+1; \\
   |O|^{1-\frac{p}{n+1}}\leq\frac{(\frac{n}{p-n-1})^{p-1}(1-e^{\frac{nR}{1-p}})^{1-p}}{(n+1)v_{n+1}^{\frac{p}{n+1}}}\int_{\partial K} \Bigl( \max\big\{1,\tfrac{H}{n}\big\} \Bigl)^{n} d\sigma_g  & {\rm if} \ p>n+1.
  \end{cases}
$$
\end{theorem}
\begin{remark}
\begin{itemize}
  \item Both the Euclidean space $\mathbb{R}^{n+1}$ and the hyperbolic space $\mathbb{H}^{n+1}$  satisfy the assumptions of the theorem. In fact, they are rotationally symmetric with $\varphi(r)=r$ and $\varphi(r)=\sinh r$ respectively. Moreover, they satisfy the isoperimetric property that geodesic balls centered at the pole are isoperimetric regions.
  \item A direct calculation indicates that
  $$
  -ng\leq\mathrm{Ric}\leq 0\Leftrightarrow
  \begin{cases}
  -n\leq-n\frac{\varphi''}{\varphi}\leq 0 \\
   -n\leq(n-1)\frac{1-\varphi'^2}{\varphi^2}-\frac{\varphi''}{\varphi}\leq 0
   \end{cases}
   \Leftrightarrow0\leq\varphi''\leq \varphi
  $$
\end{itemize}
\end{remark}

\section{Lower bounds for normalized capacity}
In this section we prove Theorem \ref{thm:1.6} by considering four cases.
\begin{itemize}
  \item [(1)] We first assume that $\operatorname{Ric}\geq0$ and consider the case $\mu=1$. Let $\lambda_{1,p}(O)$ denote the first Dirichlet eigenvalue of the $p$-Laplacian on $O$. Then
  \begin{equation}\label{eq:eigen-vs-cap}\begin{aligned}
    \lambda_{1,p}(O) &= \inf_{u \in W^{1,p}_0(O)} \frac{\int_O |\nabla u|^p dv}{\int_O |u|^p dv} \\
    &\leq \inf_{\substack{u \in W^{1,p}_0(O) \\ u|_K \equiv 1}} \frac{\int_O |\nabla u|^p dv}{\int_O |u|^p dv} \\
    &\leq \inf_{\substack{u \in W^{1,p}_0(O) \\ u|_K \equiv 1}} \frac{\int_O |\nabla u|^p dv}{|K|} \\
    &= \frac{\mathrm{Cap}_p(K,O)}{|K|}.
  \end{aligned}\end{equation}
  Combining this with the estimate in \cite{jin2026estimatedirichleteigenvalueplaplacian},
  \[
    \lambda_{1,p}(O) > (p-1)\left(\frac{\pi_p}{2 \mathrm{diam}(O)}\right)^p\ \ \&\ \ \pi_p=\frac{2\pi}{p\sin\frac{\pi}{p}},
  \]
  we obtain that
  \[
    \mathcal{C}^{(M,g)}_{p,1}(K,O) = \frac{\mathrm{Cap}_{p}(K,O)\cdot \mathrm{diam}(O)^{p}}{|K|} > (p-1)\left(\frac{\pi_p}{2}\right)^p > 0
  \]
and hence
\begin{equation}
  \inf_{(M,g) \in \mathcal{M}_0}\ \inf_{K\subseteq O\subseteq M}\mathcal{C}^{(M,g)}_{p,1}(K,O)\geq  (p-1)\left(\frac{\pi_p}{2}\right)^p > 0.
  \end{equation}
 \par In the case $\mu > 1,$ let us consider the geodesic ball $B=B(q,\mathrm{diam}(O))$ for some $q \in O$ and we have $O \subseteq B$. Since $\mathrm{Ric} \geq 0,$ the volume comparison implies that
  \begin{equation}\label{eq4.1}
    |O| \leq |B| \leq |B_{\R^{n+1}}(\mathrm{diam}(O))| = \frac{\omega_{n}}{n+1} \mathrm{diam}(O)^{n+1}.
  \end{equation}
     Here $\omega_n=|\mathbb{S}^n|$. Hence
  \begin{align*}
    \mathcal{C}^{(M,g)}_{p,\mu}(K,O) &= \mathcal{C}^{(M,g)}_{p,1}(K,O) \cdot \left(\frac{\mathrm{diam}(O)^{n+1}}{|K|}\right)^{\mu-1} \\
    &\geq (p-1)\left(\frac{\pi_p}{2}\right)^p \Big(\frac{\omega_{n}}{n+1}\Big)^{1-\mu} \\
    &= C(n,p,\mu) \\
    & > 0.
  \end{align*}
 As a consequence, \eqref{eq1.5} holds for $\mu>1.$
\par In the case  $\mu<1$, consider the warped product manifolds $M_\epsilon$ in \eqref{eq:warped}. Let $B_1$ and $B_2$ be the geodesic balls in $M_\varepsilon$ centered at the origin with radii $1$ and $2$ respectively. Denote the area of the geodesic sphere centered at the origin with radius $t$ by $S_t$. Then
  \begin{align*}
    \mathrm{Cap}_p(B_1,B_2) &= \left(\int_1^2 S_t^{\frac{1}{1-p}}\,dt\right)^{1-p} \\
    &= \omega_{n}\left(\int_1^2 j_\varepsilon^{\frac{n}{1-p}}(t)\, dt\right)^{1-p} \\
    &= \omega_{n}\varepsilon^{n}.
  \end{align*}
On the other hand,
 $$\mathrm{diam}(B_2)\in[2,4],\ \ \&\ \ |B_1|=\int_0^1 S_t\, dt\in[(1-\delta)\omega_{n}\varepsilon^{n},\omega_{n}\varepsilon^{n}].$$
Hence we obtain that
  \begin{align*}
    \mathcal{C}^{(M_\varepsilon,g_\varepsilon)}_{p,\mu}(B_1,B_2) &= \frac{\mathrm{Cap}_p(B_1,B_2)\cdot \mathrm{diam}(B_2)^{(n+1)(\mu-1)+p}}{|B_1|^\mu} \\
    &\leq \frac{\omega_{n}\varepsilon^{n}\cdot C(n,p,\mu)}{C(n,\mu)\varepsilon^{n\mu}} \\
    &= C(n,p,\mu,\omega_{n})\varepsilon^{n(1-\mu)}. \\
  \end{align*}
  Letting $\varepsilon\to0$, we obtain
   $$\inf_{\varepsilon>0} \mathcal{C}^{(M_\varepsilon,g_\varepsilon)}_{p,\mu}(B_1,B_2) = 0$$
   because $\mu<1$.

  \item  [(2)] To prove the second part of the theorem, we first show that when $\mu > 1$, for any $\kappa > 0$,
   \[
    \inf_{(M,g) \in \mathcal{M}_\kappa} \inf_{K \subseteq O \subseteq M} \mathcal{C}^{(M,g)}_{p,\mu}(K,O) = 0.
   \]
   By the scaling property, it suffices to consider $\kappa=1$. We take the hyperbolic space $\mathbb{H}^{n+1} = (\R^{n+1}, dt^2 + \sinh^2 t g_{\mathbb{S}^n}) \in \mathcal{M}_{1}$ and let $B_r$ and $B_{2r}$ be the geodesic balls in $\mathbb{H}^{n+1}$ centered at the origin with radii $r$ and $2r$ respectively. Then
   \[
    \mathrm{Cap}^{(\mathbb{H}^{n+1},g_{\mathbb{H}^{n+1}})}_p(B_r,B_{2r}) = \left( \int_r^{2r} \sinh^{\frac{n}{1-p}} t dt \right)^{1-p} \leq \left( \int_r^{2r} e^{\frac{n}{1-p} t} dt \right)^{1-p} \leq C(n,p) e^{nr},
   \]
   and
   \[
    \mathrm{diam} (B_{2r}) \leq 4r.
   \]
    On the other hand, for all sufficiently large $r$,
    \[
      |B_r| = \omega_n\int_0^r \sinh^n t\,dt \geq C(n)\int_{r/2}^r e^{nt}\,dt \geq C(n)e^{nr}.
    \]
    Therefore,
    \begin{align*}
      \mathcal{C}^{(\mathbb{H}^{n+1},g_{\mathbb{H}^{n+1}})}_{p,\mu}(B_r,B_{2r}) &= \frac{\mathrm{Cap}^{(\mathbb{H}^{n+1},g_{\mathbb{H}^{n+1}})}_p(B_r,B_{2r}) \cdot \mathrm{diam}(B_{2r})^{(n+1)(\mu-1)+p}}{|B_r|^\mu} \\
      &\leq \frac{C(n,p) e^{nr} (4r)^{(n+1)(\mu-1)+p}}{C(n,\mu) e^{n\mu r}} \\
      &= C(n,p,\mu) r^{(n+1)(\mu-1)+p} e^{n(1-\mu)r}  \to 0 \quad \text{as $r \to \infty$}.
    \end{align*}
    Now let $\mu=1$. Consider the construction in Section 3.2 of \cite{jin2026principalpfrequencyestimatesnoncompact}: Let $\psi$ be a smooth cutoff function such that $\psi(x)=0$ for $0 \leq x \leq 1/2$ and $\psi(x)=1$ for $x \geq 1$. Now let $h(x) = (1-\psi(x)) \coth x - \psi(x)$ and
    \[
      f(x) = \left\{\begin{array}{ll}
        \sinh x, & \quad 0 \leq x \leq 1/2, \\
        \sinh(1/2) \exp \int_{1/2}^x h(t) dt, & \quad x \geq 1/2.
      \end{array}\right.
    \]
    Then $f(x) \in C^\infty(\R^+)$ and satisfies
    \[
      f(x) = \left\{\begin{array}{ll}
        \sinh x, & \quad 0 \leq x \leq 1/2, \\
        C e^{-x}, & \quad x \geq 1,
      \end{array}\right.
    \]
    where $C>0$ is a constant. Let $$(M,g) = (\R^{n+1}, dr^2 + f^2(r) g_{\mathbb{S}^n})$$ and Lemma 3.1 in \cite{jin2026principalpfrequencyestimatesnoncompact} shows that $\mathrm{Ric}[g] \geq -n$.
    Consider $K = [0,R] \times \mathbb{S}^n$ and $O = [0,R+D] \times \mathbb{S}^n$. Then
  \begin{equation}\label{eq:rmk1.9}
  \begin{aligned}
    \mathrm{Cap}_p(K,O) &= \left( \int_{R}^{R+D} (\omega_n e^{-nr})^{\frac{1}{1-p}} dt \right)^{1-p} = C(n,p) e^{-nR}(e^{\frac{n}{p-1}D}-1)^{1-p}. \\
    \mathrm{diam}(O) &\leq R+D+\pi e^{-(R+D)} \leq R + D + 1. \\
    |K| &= \int_0^R \omega_n f^n(r) dr \geq \int_1^R \omega_n f^n(r) dr = C(n) (e^{-n} - e^{-nR}).
  \end{aligned}
\end{equation}
  Hence
  \begin{align*}
    \mathcal{C}^{(M,g)}_{p,1}(K,O) &= \frac{\mathrm{Cap}_p(K,O)\,\mathrm{diam}(O)^p}{|K|} \\
    &\leq C(n,p) \frac{e^{-nR}\bigl(e^{\frac{nD}{p-1}}-1\bigr)^{1-p}(R+D+1)^p}{e^{-n}-e^{-nR}} \\
    &\to 0 \qquad \text{as } R \to \infty,
  \end{align*}
      and \eqref{eq1.6} follows, since the case $\mu<1$ is contained in the first part of the proof.
  \item  [(3)] Now we prove \eqref{eq:1.7}. First let $\mu=1$.
    Fix $D_0>0$. By Corollary 1.5 in \cite{jin2026principalpfrequencyestimatesnoncompact}, every domain $O$ with $D=\mathrm{diam}(O)\leq D_0$ satisfies
    \[
      \lambda_{1,p}(O) \geq \frac{C(p,n,\kappa,D_0)}{D^p}.
    \]
    From \eqref{eq:eigen-vs-cap}, we obtain
    \[
      \mathcal{C}^{(M,g)}_{p,1}(K,O) = \frac{\mathrm{Cap}_{p}(K,O) D^p}{|K|} \geq \lambda_{1,p}(O) D^p \geq C(p,n,\kappa,D_0) > 0.
    \]
    Let $\mu>1$. By Bishop--Gromov volume comparison, for $D\leq D_0$,
    \[
      |K|\leq |O|\leq V_{-\kappa}(D)\leq C(n,\kappa,D_0)D^{n+1},
    \]
    where $V_{-\kappa}(D)$ denotes the volume of a radius-$D$ ball in the simply connected space form of sectional curvature $-\kappa$. Therefore,
    \[
      \mathcal{C}_{p,\mu}(K,O)
      =\mathcal{C}_{p,1}(K,O)
      \left(\frac{D^{n+1}}{|K|}\right)^{\mu-1}
      \geq C(p,n,\mu,\kappa,D_0)>0.
    \]
    If $\mu<1$, rescale the example $(M_\varepsilon,g_\varepsilon)$ from part (1) by $\widetilde g_\varepsilon=\lambda^2g_\varepsilon$, where $\lambda=D_0/4$. Then $\operatorname{diam}_{\widetilde g_\varepsilon}(B_2)\leq D_0$, while the scale invariance of $\mathcal{C}_{p,\mu}$ preserves the limit $\mathcal{C}_{p,\mu}(B_1,B_2)\to0$.
  \item [(4)] Finally, we prove \eqref{eq:1.8}. By scaling, it is enough to consider $\kappa=1$. If $\mu=1$,  then Corollary 1.5 in \cite{jin2026principalpfrequencyestimatesnoncompact} and \eqref{eq:eigen-vs-cap} give
  \[
      \frac{\mathrm{Cap}_p(K,O)}{|K|} \geq \lambda_{1,p}(O) \geq \bar \lambda_{\mathrm{diam}(O),p,n+1} \geq C(n,p)e^{-n \mathrm{diam}(O)},
  \]
  which is
  \[
    \mathcal{D}^{(M,g)}_{p,1}(K,O) = \frac{\mathrm{Cap}_p(K,O) e^{n \mathrm{diam}(O)}}{|K|} \geq C(n,p) > 0.
  \]
  Now let $\mu>1$, then by the volume comparison,
  \[
    |K| \leq |B_{\mathbb{H}^{n+1}}(\mathrm{diam}(K))| \leq C(n) e^{n \mathrm{diam}(K)} \leq C(n) e^{n \mathrm{diam}(O)}.
  \]
  Hence
  \begin{align*}
    \mathcal{D}_{p,\mu}(K,O) &= \mathcal{D}_{p,1}(K,O) \cdot \frac{e^{n(\mu-1) \mathrm{diam}(O)}}{|K|^{\mu-1}} \\
    &\geq C(n,\mu)\mathcal{D}_{p,1}(K,O) \\
    &\geq C(n,p,\mu)>0.
  \end{align*}
  If $\mu<1$, let $(M_\epsilon,g_\epsilon)$, $K=B_1$, $O=B_2$ be the same as in the case $\mu<1$ in the proof of the first part, i.e. let $(M_\epsilon,g_\varepsilon)$ be the same as in \eqref{eq:warped} and $K,O$ be the geodesic balls centered at the origin with radii $1$ and $2$. We have
  \begin{align*}
    \mathcal{D}_{p,\mu}(B_1,B_2) &= \frac{\mathrm{Cap}_p(B_1,B_2) e^{n\mu \mathrm{diam}(B_2)}}{|B_1|^\mu} \\
    &\leq \frac{\omega_n \epsilon^n e^{4n \mu}}{C(n,\mu) \epsilon^{n\mu}} \\
    &= C(n,p,\mu,\omega_n) \epsilon^{n(1-\mu)} \to 0 \quad \text{as $\epsilon \to 0$}.
  \end{align*}
\end{itemize}

\begin{remark}
  Let us focus on the case $M=\R^n$ and $\mu=1$. It is known that
  \[
    \mathrm{Cap}_p(K,O) \geq \mathrm{Cap}_p(B_r,B_R),
  \]
  where $B_r$ and $B_R$ are concentric balls with the same volume as $K$ and $O$ respectively, and their radii are $r$ and $R$ respectively (see \cite{maz2013sobolev}). Denote the area of the sphere with radius $t$ by $S_t$. Together with the classical isodiametric inequality $|O| \leq |B_1| (\mathrm{diam}(O)/2)^n$, we have
  \begin{align*}
    \mathcal{C}^{(\R^n,g_{\R^n})}_{p,1}(K,O) &\geq \frac{\mathrm{Cap}_p(B_r,B_R) 2^p R^p}{|B_1| r^n} \\
    &= \left(\int_r^R S_t^{\frac{1}{1-p}} dt\right)^{1-p} \frac{2^p R^p}{|B_1| r^n} \\
    &= \left(\int_r^R t^{\frac{n-1}{1-p}} dt\right)^{1-p} \frac{n2^p R^p}{r^n} \\
    &= \left\{\begin{array}{ll}
      n2^p\left(\frac{p-1}{n-p}\right)^{1-p} \left( (R/r)^{\frac{p}{1-p}}-(R/r)^{\frac{n}{1-p}} \right)^{1-p}, & \quad p < n \\
      n2^p (\ln (R/r))^{1-p} (R/r)^p, & \quad p=n, \\
      n2^p\left(\frac{p-1}{p-n}\right)^{1-p} \left( (R/r)^{\frac{n}{1-p}} - (R/r)^{\frac{p}{1-p}} \right)^{1-p}, & \quad p > n
    \end{array}\right. \\
    &\geq \left\{\begin{array}{ll}
      n 2^p (p-1)^{1-p} \left(\frac{p^p}{n^n}\right)^{\frac{1-p}{n-p}}, & \quad p \neq n \\
      p 2^p \left(\frac{p-1}{p}\right)^{1-p} e^{p-1}, & \quad p=n. \\
    \end{array}\right.
  \end{align*}
  Equality holds if $K$ and $O$ are concentric balls of radii $r$ and $R$, respectively, and
  \[
    R/r = \left\{\begin{array}{ll}
      (p/n)^{\frac{1-p}{n-p}}, & \quad p \neq n \\
      e^{\frac{p-1}{p}}, & \quad p=n.
    \end{array}\right.
  \]
  It is straightforward to check that the right-hand side of the above lower bound is continuous with respect to $p$.
\end{remark}

\begin{proof}[Proof of Remark \ref{rmk1.7}]

  Let $\epsilon > 0$ be sufficiently small and let $\eta_\epsilon(t)$ be a non-increasing $C^\infty$-function on $[0,+\infty)$, such that $\eta_\epsilon(t) = 1$, $\eta_\epsilon|_{[\delta,+\infty)} \equiv 0$ for some $\delta = \delta(\epsilon) = \mathcal{O}(\epsilon)$ as $\epsilon \to 0$, and $\int_0^\infty \eta_\epsilon(t) dt = \epsilon$. Then the function
  \begin{equation}\label{eq1.7}
    j_\epsilon(t) = \int_0^t \eta_\epsilon(s) ds
  \end{equation}
  is smooth and non-decreasing and satisfies $j_\epsilon(0)=0$, $j_\epsilon'(0)=1$, $j_\epsilon|_{[\delta,+\infty)} \equiv \epsilon$. Consider the warped product metric
  \begin{equation}\label{eq:warped}
    M_\epsilon = (\R^{n+1}, g_\epsilon = dt^2 + j_\epsilon^2 g_{\mathbb{S}^n}).
  \end{equation}
  It satisfies $\mathrm{Ric} \geq 0$ (see \cite{jin2026estimatedirichleteigenvalueplaplacian}).

  If $b > 1-p$. Consider the geodesic balls $B = B(o,2\delta)$, $B' = B(o,2 \delta + \delta^{1/2})$. By similar calculations as above, we have
  \begin{align*}
    \mathcal{C}^{(M_\epsilon,g_\epsilon)}_{p,a,b,c,d}(B,B') &= \frac{\mathrm{Cap}_p(B,B')}{|B|^a|B'|^b|\partial B|^c |\partial B'|^d} \\
    &\leq \frac{\left(\int_{2 \delta}^{2 \delta + \delta^{1/2}} (\epsilon^n \omega_n)^{\frac{1}{1-p}} dt\right)^{1-p}}{\left( \int_\delta^{2 \delta} \epsilon^n \omega_n dt \right)^a \left( \int_\delta^{2 \delta + \delta^{1/2}} \epsilon^n \omega_n dt \right)^b (\epsilon^n \omega_n)^{c+d}} \\
    &\leq C \frac{\delta^{\frac{1-p}{2}}\epsilon^n}{\delta^{a+\frac{b}{2}} \epsilon^{n(a+b+c+d)}} \\
    &\leq C \epsilon^{(b-1+p)/2}.
  \end{align*}
  The right-hand side tends to 0 as $\epsilon \to 0$.

  Now assume $a+b \neq 1-p$. Consider the manifold $(\mathbb{S}^1 \times \R^n, \epsilon^2 g_{\mathbb{S}^1} \times g_{\R^n})$ and $K = \mathbb{S}^1 \times B_r$, $O = \mathbb{S}^1 \times B_R$, where $0 < r < R$. Then
  \[
    \mathrm{Cap}_p(K,O) = \left(\int_r^R |\mathbb{S}^1 \times \partial B_t|^{\frac{1}{1-p}} dt \right)^{1-p} = C(n,p,r,R)\epsilon,
  \]
  and
  \[
    |K||O||\partial K||\partial O| = C(n,r,R,a,b,c,d)\epsilon^{a+b+c+d}.
  \]
  Hence we obtain
  \begin{align*}
    \mathcal{C}^{(M_\epsilon,g_\epsilon)}_{p,a,b,c,d}(K,O) &= \frac{C(n,p,r,R)\epsilon}{C(n,r,R,a,b,c,d) \epsilon^{a+b+c+d}} \\
    &= C(n,p,r,R,a,b,c,d) \epsilon^{1-(a+b+c+d)} \\
    &= C(n,p,r,R,a,b,c,d) \epsilon^{\frac{p-1+a+b}{n}}.
  \end{align*}
  If $a+b > 1-p$, then the right-hand side tends to $0$ as $\epsilon \to 0$; and if $a+b < 1-p$, then the right-hand side also tends to $0$ as $\epsilon \to \infty$.
\end{proof}


\noindent{Xiaoshang Jin}\\
  School of mathematics and statistics, Huazhong University of science and technology, 430074, Wuhan, P.R. China
 \\Email address: jinxs@hust.edu.cn
 \\~\\
 \noindent{Zhiwei L\"u}\\
  School of mathematics and statistics, Huazhong University of science and technology, Wuhan, P. R. China. 430074
 \\Email address: m202470005@hust.edu.cn
 \\~\\
\noindent{Jiabin Yin}\\
 School of Mathematics and Statistics, Xinyang Normal University, 464000, Xinyang, P.R. China
	\\Email address: jiabinyin@126.com
\end{document}